\documentclass[a4paper,12pt]{amsart}
\usepackage{amssymb}
\usepackage{ifthen}
\usepackage[dvips]{graphicx}
\usepackage{cite}
\nonstopmode \numberwithin{equation}{section}
\usepackage{amssymb}
\usepackage{ifthen}
\usepackage{graphicx}
\usepackage{amsmath}
\usepackage[T1]{fontenc} 
\usepackage[utf8]{inputenc}
\usepackage[usenames,dvipsnames]{color}
\usepackage{color}
\usepackage[english]{babel}
\usepackage{fancyhdr}
\usepackage{fancybox}
\usepackage{tikz}

\nonstopmode \numberwithin{equation}{section}
\theoremstyle{plain}

\newtheorem{prop}{Proposition}

\newtheorem{conj}{Conjecture}

\theoremstyle{definition}
\newtheorem{defn}{Definition}[section]

\newtheorem{thm}{Theorem}[section]
\newtheorem{cor}{Corollary}[section]
\newtheorem{lem}{Lemma}[section]
\newtheorem{prob}{Problem}
\newtheorem{rem}{Remark}[section]

\newcounter{minutes}
\divide\time by 60
\newcounter{hours}
\multiply\time by 60
\addtocounter{minutes}{-\time}

\newcounter {own}
\def\theown {\thesection       .\arabic{own}}

\newenvironment{pf}[1][]{%
 \vskip 3mm
 \noindent
 \ifthenelse{\equal{#1}{}}%
  {{\slshape Proof. }}%
  {{\slshape #1.} }%
 }%
{\qed\bigskip}

\newcounter{alphabet}

\newcommand{\Aut}{{\operatorname{Aut}}}

\def\be{\begin{equation}}
\def\ee{\end{equation}}

\newcommand{\bee}{\begin{enumerate}}
\newcommand{\eee}{\end{enumerate}}

\newcommand{\blem}{\begin{lem}}
\newcommand{\elem}{\end{lem}}
\newcommand{\bthm}{\begin{thm}}
\newcommand{\ethm}{\end{thm}}
\newcommand{\bcor}{\begin{cor}}
\newcommand{\ecor}{\end{cor}}
\newcommand{\beg}{\begin{examp}}
\newcommand{\eeg}{\end{examp}}
\newcommand{\begs}{\begin{examples}}
\newcommand{\eegs}{\end{examples}}

\newcommand{\bdefn}{\begin{defn}}
\newcommand{\edefn}{\end{defn}}

\newcommand{\bprob}{\begin{prob}}
\newcommand{\eprob}{\end{prob}}
\newcommand{\bei}{\begin{itemize}}
\newcommand{\eei}{\end{itemize}}

\newcommand{\bcon}{\begin{conj}}
\newcommand{\econ}{\end{conj}}
\newcommand{\bcons}{\begin{conjs}}
\newcommand{\econs}{\end{conjs}}
\newcommand{\bprop}{\begin{prop}}
\newcommand{\eprop}{\end{prop}}
\newcommand{\br}{\begin{rem}}
\newcommand{\er}{\end{rem}}
\newcommand{\brs}{\begin{rems}}
\newcommand{\ers}{\end{rems}}
\newcommand{\bo}{\begin{obser}}
\newcommand{\eo}{\end{obser}}
\newcommand{\bos}{\begin{obsers}}
\newcommand{\eos}{\end{obsers}}
\newcommand{\bpf}{\begin{pf}}
\newcommand{\epf}{\end{pf}}
\newcommand{\ba}{\begin{array}}
\newcommand{\ea}{\end{array}}
\newcommand{\beq}{\begin{eqnarray}}
\newcommand{\beqq}{\begin{eqnarray*}}
\newcommand{\eeq}{\end{eqnarray}}
\newcommand{\eeqq}{\end{eqnarray*}}

\begin{document}

\title{Normality criterion for logharmonic mapping}

\author{Molla Basir Ahamed}
\address{Molla Basir Ahamed,
	Department of Mathematics,
	Jadavpur University,
	Kolkata-700032, West Bengal, India.}
\email{mbahamed.math@jadavpuruniversity.in}

\author{Sanju Mandal}
\address{Sanju Mandal,
	Department of Mathematics,
	Jadavpur University,
	Kolkata-700032, West Bengal, India.}
\email{sanjum.math.rs@jadavpuruniversity.in, sanju.math.rs@gmail.com}

\subjclass[{AMS} Subject Classification:]{Primary: 30C45; 30C55; 31A05}

\def\thefootnote{}
\footnotetext{ {\tiny File:~\jobname.tex,
printed: \number\year-\number\month-\number\day,
          \thehours.\ifnum\theminutes<10{0}\fi\theminutes }
} \makeatletter\def\thefootnote{\@arabic\c@footnote}\makeatother

\begin{abstract}
In this paper, we present several necessary and sufficient conditions for a logharmonic mapping to be normal \textit{i.e.}, we establish Marty’s criterion, Zalcman–Pang lemma and the Lohwater–Pommerenke theorem for logharmonic mappings, along with an application of the Zalcman–Pang lemma.
\end{abstract}

\maketitle
\pagestyle{myheadings}
\markboth{M. B. Ahamed and S. Mandal}{Normality criterion for logharmonic mapping}

\section{\bf Introduction}
Let $\mathcal{A}(\mathbb{D})$ be the class of linear space of all analytic functions defined in the open unit disk $\mathbb{D} := \{z \in \mathbb{C} : |z| < 1\}$. A complex-valued function $f$ is said to be harmonic if both $\mathrm{Re}\{f\}$ and $\mathrm{Im}\{f\}$ are real harmonic. In other words, harmonic functions $f$ are the solutions of $\Delta f = 0$, where $\Delta$ is the Laplacian operator and is defined as 
\begin{align*}
	\Delta =4\frac{\partial^2}{\partial z\partial\overline{z}} =\frac{\partial^2}{\partial x^2} + \frac{\partial^2}{\partial y^2}.
\end{align*}
Every harmonic mapping f has crucial property, that is, it has canonical decomposition $f=h+\overline{g}$, where $h,g\in\mathcal{A} (\mathbb{D})$ and respectively known as analytic and co-analytic
part of $f$, this representation is unique with the condition $g(0)=0$. In $ 1936 $, Lewy \cite{Lewy-BAMS-1936} obtained a necessary and sufficient condition for a complex-valued harmonic mapping $f=h+\overline{g} $ is locally univalent and sense preserving in $ \mathbb{D} $ is that the $ h^{\prime}(z)\neq 0 $ and the Jacobian $ J_f(z) $ is positive in $ \mathbb{D} $, where $ J_f(z)= |h^{\prime}(z)|^2 - |g^{\prime}(z)|^2 $. Denote $\mathcal{H} (\mathbb{D})$, the class of all complex-valued harmonic mappings defined in $\mathbb{D}$.\vspace{1.5mm}

A mapping $f$ defined in $\mathbb{D}$ is logharmonic if $\log(f)\in \mathcal{H}(\mathbb{D})$. Alternatively, the logharmonic mappings are the solutions of the non-linear elliptic partial differential equation
\begin{align}\label{eq-1.1}
	\frac{\overline{f}_{\overline{z}}}{\overline{f}}=\omega \frac{f_z}{f},
\end{align}
where $\omega\in\mathcal{A}(\mathbb{D})$ with the condition $|\omega|<1$ and is called the second dilatation of $f$. Note that if $f_1$ and $f_2$ are two logharmonic functions with respect to $\omega$, then $f_1 f_2$ is logharmonic with respect to same $\omega$, and $f_1/f_2$ is logharmonic (provided $f_2\neq 0$). The logharmonicity preserves pre-composition with a conformal mapping, whereas it is not always true for a post-composition. Furthermore, logharmonicity is not invariant under translation and inversion. Every non-constant logharmonic mapping is quasi-regular, therefore it
is continuous and open. The logharmonic mapping $f$ is also sense preserving as its Jacobian $J_f(z)= |f_z(z)|^2 - |f_{\overline{z}}(z)| ^2$ is positive. The salient properties like the modified Liouville’s theorem, the maximum principle, the identity principle and the argument principle hold true for logharmonic mappings \cite{Abdulhadi-Ali-AAA-2012,Meng-Ponnusamy-Qiao-JGA-2024}. \vspace{1.5mm}

It has been shown that if $f$ is a non-vanishing logharmonic mapping then $f$ admits the following representation
\begin{align}\label{eq-1.2}
	f(z)=h(z)\overline{g(z)},
\end{align}
where $h(z), g(z)$ are analytic in the unit disk $\mathbb{D}$ with $h(0)\neq 0$, $g(0)=1$ \cite{Abdulhadi-Bshouty-TAMS-1988}. If $f$ vanishes at $z =0$ but it is not identically zero, then $f$ can be represented by
\begin{align}\label{eq-1.3}
	f(z)=z|z|^{2\beta}h(z)\overline{g(z)},
\end{align}
where $\mathrm{Re}\{\beta\}>- \frac{1}{2}$, $h(z)$ and $g(z)$ are analytic in the unit disk $\mathbb{D}$ with $h(0)\neq 0$ and $g(0)=1$. The class of sense-preserving logharmonic mappings is denoted by $S_{LH}$. \vspace{1.5mm}

Let $f(1)=1$, and $\mathrm{Re}\{\beta\}>-\frac{1}{2}$, then the function $f_{\beta}(z)=z|z|^{2\beta}$ is a solution of the equation \eqref{eq-1.1} in $\mathbb{C}$ with $\omega=\frac{\overline{\beta}} {1+\beta}$. It is a simple way that $f$ maps the unit disk $\mathbb{D}$ onto itself, since $\omega\in\mathcal{A}(\mathbb{D})$ and $|\omega|<1$, the Jacobian
\begin{align*}
	J_f(z)= |f_z(z)|^2\left(1-|\omega(z)|^2\right)
\end{align*}
is positive and all non-constant logharmonic mappings are sense-preserving in the open unit disk $\mathbb{D}$. We remark that the composition $f\circ\varphi$ of a logharmonic mapping $f$ with an analytic mapping $\varphi$ is also logharmonic. However, the composition $\varphi\circ f$ of an analytic mapping $\varphi$ with
a logharmonic mapping $f$ is in general not logharmonic. The class of univalent logharmonic mappings has been investigated extensively by many works \cite{Abdulhadi-Bshouty-TAMS-1988,Abdulhadi-IJM-1996, Abdulhadi-IJM-2002}. \vspace{1.5mm}

A function $ f $ meromorphic in $ \mathbb{D} $ is called a normal function if the family $\mathfrak{F}=\{f\circ\phi: \phi\in \Aut(\mathbb{D})\}$ is a normal family, where $\Aut(\mathbb{D})$ denotes the class of conformal automorphisms of $\mathbb{D}$. Normal functions play important roles in studying the properties of meromorphic functions, especially their behavior at the boundary.
\begin{defn}\cite{Noshiro-JFSHU-1938,Yosida-PPMSJS-1934}
A meromorphic function $f$ is said to normal if, and only if, 
\begin{equation}\label{eq-1.4}
	\sup_{z\in\mathbb{D}}\left(1-|z|^2\right)f^{\#}(z)<\infty
\end{equation} 
where $ f^{\#} $ denotes the spherical derivative of $ f $ given by $f^{\#}(z)= {|f^{\prime}(z)|}/{(1+|f(z)|^2)}$.
\end{defn}
The condition \eqref{eq-1.4} is equivalent to say that $f$ is Lipschitz when considered as a mapping from the hyperbolic disc into the extended complex plane $\hat{\mathbb{C}}:=\mathbb{C}\cup \{\infty\}$ with the chordal distance. The normal meromorphic functions include two important examples that univalent meromorphic functions and analytic functions which omit two values. Their importance were shown in geometric functions theory and the behaviour
in the boundary of meromorphic functions in the unit disk $\mathbb{D}$. Many scholars have studied the properties of normal meromorphic functions from both analytic and geometric points of view \cite{Yamashita-MZ-1975,Pommerenke-MMJ-1974,Clunie-Anderson-Pommerenke-JRM-1974}. \vspace{1.5mm}

The real valued harmonic functions were defined in $\mathbb{D}$, in [14] Lappan established that $\varphi$ is normal if
\begin{align*}
	\sup_{z\in\mathbb{D}}(1-|z|^2) \frac{|grad\;\varphi(z)|}{1+ |\varphi(z)|^2}<\infty,
\end{align*}
where $grad\;\varphi(z)$ is the gradient vector of $\varphi$. He showed that if $\varphi$ is a harmonic normal function, and if $f=\varphi +i\phi$ is analytic in $\mathbb{D}$, then $f$ is normal function. Since the topic of harmonic mappings of complex value is one of the most studied in the domain of complex analysis nowadays. It seems natural to address the issue of normal harmonic mappings of complex value defined in the unit disk, we remark that an important concept related with normal harmonic functions is the Bloch function, which was studied by Colonna in \cite{Colonna-IUMJ-1989}. Now, we recall here the definition of normal harmonic mapping as follows.
\begin{defn}\cite{Arbe-Hern-MM-2019,Deng-Ponn-Qiao-MM-2020}
A harmonic mapping $ f=h+\overline{g} $ in $ \mathbb{D} $ is said to be normal if $\sup_{z\in\mathbb{D}}\left(1-|z|^2\right)f^{\#}(z)< \infty$, where 
\begin{align*}
	f^{\#}(z) = \frac{|h^{\prime}(z)|+ |g^{\prime}(z)|}{1 + |f(z)|^2}.
\end{align*} 
\end{defn}
Following the idea of Colonna \cite{Colonna-IUMJ-1989} on harmonic Bloch functions, Arbeláez \textit{et al.} \cite{Arbe-Hern-MM-2019} studied normal harmonic mappings and established some necessary conditions for a harmonic mapping to be normal. Moreover, in \cite{Deng-Ponn-Qiao-MM-2020,Ahamed-Mandal-MM-2023,Arbe-Hern-MM-2019}, the authors have studied some properties of normal harmonic mappings. \vspace{1.5mm}

Planar logharmonic mappings constitute an important tool in the study of harmonic mappings from the point of view of geometric function theory. In $2020$, Jiang \cite{Jiang-HJM-2020} studied the normality criterian of logharmonic mappings and prove that a logharmonic mapping $f$ defined in the unit disk $\mathbb{D}$ is normal if it satisfies a Lipschitz type condition. Here, we bring to attention the definition of a normal logharmonic mapping as follows.
\begin{defn}\cite{Jiang-HJM-2020}
A logharmonic mapping $f(z)=z|z|^{2\beta}h(z)\overline{g(z)}$ in the unit disk $\mathbb{D}$ is said to be normal if and only if $\sup_{z\in\mathbb{D}}(1-|z|^2) f^{\#}(z)<\infty$, where
\begin{align*}
	f^{\#}(z)=\frac{|f_z(z)|+|f_{\overline{z}}(z)|}{1+|f(z)|^2}.
\end{align*}
\end{defn}
We acquaint with the chordal distance on the generalized complex plane $\hat{\mathbb{C}}$ which is defined as follows: The \textit{chordal distance} $\chi(z_1, z_2)$ between the complex values $z_1$ and $z_2$, considered
as points on the Riemann sphere, is given by
\begin{align}\label{eq-1.5}
	\chi(z_1, z_2):=
	\begin{cases}
		0\;\;\;\;\;\;\;\;\;\;\;\;\;\;\;\;\;\;\;\;\;\;\;\;\;\;\;\;\;\;\;\;\; \mbox{if}\; z_1=z_2,\\
		\dfrac{|z_1-z_2|}{\sqrt{1+|z_1|^2}\sqrt{1+|z_2|^2}} \;\;\mbox{if}\; z_1\neq \infty\neq z_2\\
		\dfrac{1}{\sqrt{1+|z_1|^2}} \;\;\;\;\;\;\;\;\;\;\;\; \;\;\;\;\;\; \mbox{if}\; z_1\neq \infty =z_2
	\end{cases}
\end{align}
If $P_{z_1}, P_{z_2}$ are the two points on the Riemann sphere, under stereographic projection, corresponding to $z_1$ and $z_2$ respectively, we have
\begin{align*}
	|P_{z_1} -P_{z_2}|=\chi(z_1, z_2).
\end{align*}
Hence
\begin{align*}
	\chi(z_1, z_2)\leq \varrho(z_1,z_2)\leq L(\Gamma),
\end{align*}
where $\varrho(z_1,z_2)$ is the spherical distance of $z_1$ and $z_2$, $\Gamma$ is any rectifiable curve in $\mathbb{C}$ with endpoints $z_1,z_2$, and
\begin{align*}
	L(\Gamma)=\int_{\Gamma} \frac{|du|}{1+|u|^2}
\end{align*}
is the spherical length of $\Gamma$. \vspace{1.5mm}

A logharmonic mapping $f:\mathbb{D}\rightarrow\mathbb{C}$ is called a normal logharmonic mapping, if
\begin{align*}
	\sup_{z_1\neq z_2} \frac{\chi(f(z_1), f(z_2))}{\rho(z_1,z_2)} < \infty,
\end{align*}
where $\rho(z_1,z_2)$ denotes the hyperbolic distance between two points $z_1$ and $z_2$ in $\mathbb{D}$, that is,
\begin{align*}
	\rho(z_1,z_2)=\frac{1}{2}\log\left(\frac{1+r}{1-r}\right), \;\;\;\;
	\mbox{where}\;\;\;\; r=\vline\frac{z_1 -z_2}{1-\bar{z_1}z_2}\vline.
\end{align*}

In this paper, we present a range of significant results that not only build upon earlier studies but also introduce new ideas that, to the best of our knowledge, are novel. Our main results provide necessary and sufficient conditions for a logharmonic mapping to be normal. In the next section, we first establish Marty’s criterion for logharmonic mappings, followed by a detailed discussion of the Zalcman–Pang lemma and the Lohwater–Pommerenke theorem for logharmonic mappings, along with an application of the Zalcman–Pang lemma.

\section{\bf Main results}
 We begin this section with the following Theorem which is a generalization of the corresponding one for analytic functions due to Marty \cite[p. 226, Theorem 17]{Ahlfors-1979}.
\begin{thm}\label{th-2.1}
A class $\mathcal{F}$ of logharmonic mapping $f(z)=z|z|^{2\beta}h(z) \overline{g(z)}$ in the unit disk $\mathbb{D}$ is normal if $\{f^{\#}(z):f\in\mathcal{F}\}$ uniformly locally bounded. (where $f^{\#}$ is defined in Definition 1.1)
\end{thm}
\begin{proof}
Consider $\chi(f(z_1), f(z_2))$ defined as in \eqref{eq-1.5} for $f(z_1)\neq\infty\neq f(z_2)$. It is easy to see that, followed by the stereographic projection, $f$ maps an arc $\gamma$ on an
image with length
\begin{align*}
	L(\Gamma)=\int_{\Gamma} \frac{|df(z)|}{1+|f(z)|^2} \leq \int_{\Gamma} \frac{\left(|f_z(z)|+ |f_{\overline{z}}(z)| \right)|dz|} {1+|f(z)|^2}=\int_{\Gamma} f^{\#}(z) |dz|.
\end{align*}
Assume that $\{f^{\#}(z):f\in\mathcal{F}\}$ uniformly locally bounded. So, there exists $M>0$ such that $f^{\#}(z)\leq M$ on the segment between $z_1$ and $z_2$, where $M$ is independent of $f$, then we have
\begin{align*}
	\chi(f(z_1), f(z_2))\leq L(\Gamma)\leq \int_{\Gamma}f^{\#}(z) |dz|
	\leq M\int_{\Gamma} |dz|=M|z_1 -z_2|,
\end{align*}
which implies that logharmonic mappings in $\mathcal{F}$ are equicontinuous when $f^{\#}(z)$’s are locally bounded. By Arzelà-Ascoli theorem \cite[p. 222, Theorem 14]{Ahlfors-1979}, the class $\mathcal{F}$ is normal.
\end{proof}

Now, we establish a necessary and sufficient condition for a logharmonic mapping in  $\mathbb{D}$ to be normal, which serves as the logharmonic counterpart to the well-known Zalcman–Pang lemma \cite{Hua-MM-1995,Pang-Zalcman-BLMS-2000}.
\begin{thm}\label{th-2.2}
A non-constant function $f(z)=z|z|^{2\beta}h(z) \overline{g(z)}$ logharmonic mapping in $\mathbb{D}$ is normal if and only if for each $\alpha\in(-1,\infty)$, there do not exist sequences of points $\{z_n\}\subset\mathbb{D}$ and of real numbers $\{\rho_n\}$ with $\rho_n>0$ and $\rho_n\rightarrow 0$ as $n\rightarrow\infty$ such that the functions 
\begin{align*}
	F_n(\xi):=\rho^{\alpha}_n f(z_n+\rho_n\xi)
\end{align*}
converges locally uniformly in $\mathbb{C}$ to a non-constant logharmonic mapping $F(\xi)$ satisfying $F^{\#}(\xi)\leq F^{\#}(0)= 1$.
\end{thm}
\begin{proof}
Suppose that $f$ is not normal. Then there exists a sequence $\{z^*_n\}\subset\mathbb{D}$ with $z^*_n\rightarrow 1$ as $n\rightarrow\infty$ such that 
\begin{align*}
	\left(1-|z^*_n|^2\right)\frac{|f_z(z^*_n)|+|f_{\overline{z}}(z^*_n)|}{1+|f(z^*_n)|^2}\rightarrow\infty \;\;\;\; \mbox{as}\;\;\;\;n\rightarrow\infty.
\end{align*}
Let $\{r_n\}$ be a sequence such that $|z^*_n|<r_n<1$ and
\begin{align}\label{eq-2.1}
	\left(1-\frac{|z^*_n|^2}{r^2_n}\right)\frac{|f_z(z^*_n)|+|f_{\overline{z}}(z^*_n)|}{1+|f(z^*_n)|^2}\rightarrow\infty \;\;\;\; \mbox{as}\;\;\;\;n\rightarrow\infty.
\end{align}
Now, we defined
\begin{align*}
	H_n(t,z):=\frac{\left\{\left(1-\frac{|z^*_n|^2}{r^2_n}\right)t\right\}^{1+\alpha}\left(|f_z(z)|+|f_{\overline{z}}(z)|\right)}{1 + \left\{\left(1-\frac{|z^*_n|^2}{r^2_n}\right)t\right\}^{2\alpha} |f(z)|^2},
\end{align*}
where $0<t\leq 1$ and $|z|<r_n$. Clearly, the functions $H_n(t,z)$ are continuous in $(0,1]\times\{|z|<r_n\}$. As $\alpha>-1$ and $f$ is logharmonic, it follows that
\begin{align}\label{eq-2.2}
	\lim_{t\rightarrow 0} H_n(t,z)=0.
\end{align}
On the other hand, it is easy to verify that
\begin{align}\label{eq-2.3}
	H_n(t,z)\geq t^{1+|\alpha|}\left(1-\frac{|z|^2}{r^2_n} \right)^{1+|\alpha|}  f^{\#}(z).
\end{align}
By \eqref{eq-2.1} and \eqref{eq-2.3}, we see that
\begin{align}\label{eq-2.4}
	H_n(1,z^*_n)\geq \left(1-\frac{|z^*_n|^2}{r^2_n} \right)^{1 +|\alpha|} f^{\#}(z)\rightarrow\infty \;\;\;\;\mbox{as} \;\;\;\;n\rightarrow \infty.
\end{align}
Therefore, from \eqref{eq-2.1}, \eqref{eq-2.2} and \eqref{eq-2.3}, we deduce that
\begin{align*}
	\sup_{|z|<r_n} H_n(1,z)\geq H_n(1,z^*_n)>1,\;\;\;\;\mbox{for sufficiently large} \;\; n.
\end{align*}
and
\begin{align*}
	\sup_{|z|<r_n} H_n(t,z)< 1,\;\;\;\;\mbox{for sufficiently small} \;\; t.
\end{align*}
Thus, there exists $t_n$ and $z_n$ with $0<t_n\leq 1$ and $|z_n|<r_n$ such that
\begin{align}\label{eq-2.5}
	\sup_{|z|<r_n} H_n(t_n,z)= H_n(t_n,z_n)=1.
\end{align}
Combining this with \eqref{eq-2.3} and \eqref{eq-2.5}, we obtain
\begin{align*}
	1=H_n(t_n,z_n)\geq H_n(t_n,z^*_n)\geq t^{1+|\alpha|}_n\left(1- \frac{|z^*_n|^2}{r^2_n}\right)^{1+|\alpha|}  f^{\#}(z^*_n)> t^{1+|\alpha|}_n H_n(1,z^*_n).
\end{align*}
Since the second factor $H_n(1,z^*_n)$ on the right-hand side tends to infinity by \eqref{eq-2.4}, it follows that $\lim_{n\rightarrow\infty} t_n=0$.\vspace{2mm}

Now set $\rho_n:=\left(1- \frac{|z_n|^2}{r^2_n}\right)t_n$, so that
\begin{align}\label{eq-2.6}
	\frac{\rho_n}{1-|z_n|}\leq \frac{\rho_n}{r_n -|z_n|}= \frac{(r^2_n -|z_n|^2)}{r^2_n(r_n -|z_n|)}t_n= \frac{(r_n +|z_n|)}{r^2_n} t_n \rightarrow 0 \;\;\;\;\mbox{as} \;\;\;\;n\rightarrow \infty.
\end{align}
Then the functions $F_n(\xi)=\rho^{\alpha}_n f(z_n+\rho_n\xi)$ is defined for $|\xi|<R_n=\frac{1-|z_n|}{\rho_n}$. Moreover, a simple computation yields
\begin{align}\label{eq-2.7}
	F^{\#}_n(\xi)= \frac{\rho^{1+\alpha}_n\left(|f_z(z_n+\rho_n\xi)| + |f_{\overline{z}}(z_n+\rho_n\xi)|\right)}{1+\rho^{2\alpha}_n |f(z_n+\rho_n\xi)|^2}.
\end{align}
In view of \eqref{eq-2.5} and \eqref{eq-2.7}, we obtain
\begin{align}\label{eq-2.8}
	F^{\#}_n(0) \nonumber&= \frac{\rho^{1+\alpha}_n\left(|f_z(z_n)| + |f_{\overline{z}}(z_n)|\right)}{1+\rho^{2\alpha}_n |f(z_n)|^2}\\&= \frac{\left\{\left(1- \frac{|z_n|^2}{r^2_n}\right)t_n \right\} ^{1+\alpha}\left(|f_z(z_n)| + |f_{\overline{z}}(z_n)|\right)} {1+\left\{\left(1- \frac{|z_n|^2}{r^2_n}\right)t_n \right\}^{2\alpha} |f(z_n)|^2} = H_n(t_n,z_n)=1.
\end{align}
It follows from \eqref{eq-2.6} that $\left(1- \frac{|z_n|^2}{r^2_n} \right)/\left(1- \frac{|z_n+ \rho_n\xi|^2}{r^2_n}\right)$ tends uniformly to $1$ as $n\rightarrow\infty$ on compact subsets of $\mathbb{C}$. Let us choose $\epsilon_n>0$ with $\epsilon_n\rightarrow 0$ as $n\rightarrow\infty$ such that
\begin{align}\label{eq-2.9}
	(1-\epsilon_n)\left(1- \frac{|z_n+ \rho_n\xi|^2}{r^2_n}\right)t_n \leq\rho_n\leq (1+\epsilon_n)\left(1- \frac{|z_n+ \rho_n\xi|^2} {r^2_n}\right)t_n.
\end{align}
Therefore, from \eqref{eq-2.5}, \eqref{eq-2.7} and \eqref{eq-2.9}, we obtain
\begin{align}\label{eq-2.10}
	F^{\#}_n(\xi)\nonumber&\leq \frac{(1+\epsilon_n)^{1+\alpha} \left\{\left(1- \frac{|z_n+ \rho_n\xi|^2} {r^2_n}\right)t_n \right\}^{1+\alpha} \left(|f_z(z_n+\rho_n\xi)| + |f_{\overline{z}} (z_n+\rho_n\xi)| \right)}{1+(1-\epsilon_n) ^{2\alpha} \left\{\left(1- \frac{|z_n+ \rho_n\xi|^2}{r^2_n}\right) t_n\right\}^{1+\alpha} |f(z_n+ \rho_n\xi)|^2} \\&\nonumber\leq \frac{(1+\epsilon_n)^{1+\alpha}\left\{\left(1- \frac{|z_n+ \rho_n\xi|^2} {r^2_n}\right)t_n\right\}^{1+\alpha} \left(|f_z(z_n +\rho_n\xi)| + |f_{\overline{z}}(z_n+\rho_n\xi)| \right)}{(1- \epsilon_n)^{2\alpha} \left\{\left(1- \frac{|z_n+ \rho_n\xi|^2} {r^2_n}\right)t_n\right\}^{1+\alpha} |f(z_n+ \rho_n\xi)|^2} \\&\nonumber=\frac{(1+\epsilon_n)^{1+\alpha}}{(1-\epsilon_n)^{2\alpha}} \cdot H_n(t_n, z_n+\rho_n\xi) \cdot \frac{1+ \left\{\left(1- \frac{|z_n+ \rho_n\xi|^2}{r^2_n}\right)t_n\right\}^{1+\alpha} |f(z_n+ \rho_n\xi)|^2}{\left\{\left(1- \frac{|z_n+ \rho_n\xi|^2} {r^2_n}\right)t_n\right\}^{1+\alpha} |f(z_n+ \rho_n\xi)|^2} \\&\leq\frac{(1+\epsilon_n)^{1+\alpha}} {(1-\epsilon_n)^{2\alpha}} \cdot H_n(t_n, z_n+\rho_n\xi) \rightarrow 1 \;\;\;\;\mbox{as} \;\;\;\;n\rightarrow \infty.
\end{align}
Thus, by Marty's Theorem \ref{th-2.1}, we conclude that the sequence $\{F_n(\xi)\}$ is normal. We may assume that $\{F_n(\xi)\}$ converges locally uniformly in $\mathbb{C}$. Then, the limit function $\{F(\xi)\}$ is logharmonic in $\mathbb{C}$. In view of \eqref{eq-2.8}
and \eqref{eq-2.10} it follows that $F^{\#}(\xi)\leq F^{\#}(0)= 1\neq 0$, which implies that $F$ is a non-constant logharmonic mapping. \vspace{2mm}

Next, we prove the necessary part of the theorem. Let $f$ be normal in $\mathbb{D}$. Again, we recall that the functions $F_n(\xi)$ given by
\begin{align*}
	F_n(\xi)=\rho^{\alpha}_n f(z_n+\rho_n\xi)
\end{align*}
are defined for $|\xi|<\frac{1-|z_n|}{\rho_n}$. Again, we observe that
\begin{align*}
	\frac{\rho^{1+\alpha}_n}{1-|z_n|}\leq \frac{\rho^{1+\alpha}_n}{r_n -|z_n|}&= \frac{(r^2_n -|z_n|^2)^{1+\alpha}}{r^{2(1+\alpha)}_n(r_n -|z_n|)}t^{1+\alpha}_n \\&= \frac{(r_n +|z_n|)^{1+\alpha}(r_n -|z_n|)^{\alpha}}{r^{2(1+\alpha)}_n} t^{1+\alpha}_n \\&< \frac{r^{\alpha}_n(r_n +|z_n|)^{1+\alpha}}{r^{2(1+\alpha)}_n} t^{1+\alpha}_n \\&= \frac{(r_n +|z_n|)^{1+\alpha}}{r^{\alpha +2}_n} t^{1+\alpha}_n \\&< \frac{2^{1+\alpha}}{r_n} t^{1+\alpha}_n \rightarrow 0 \;\;\;\;\mbox{as} \;\;\;\;n\rightarrow \infty,
\end{align*}
which implies that
\begin{align*}
	\frac{\rho^{1+\alpha}_n}{1-|z_n|-\rho_n |\xi|}\rightarrow 0 \;\;\;\;\mbox{as} \;\;\;\;n\rightarrow \infty,\;\;\mbox{for}\;\; |\xi|<\frac{1-|z_n|}{\rho_n}.
\end{align*}
Since
\begin{align*}
	F^{\#}_n(\xi)&= \frac{\rho^{1+\alpha}_n\left(|f_z(z_n+\rho_n\xi)| + |f_{\overline{z}}(z_n+\rho_n\xi)|\right)}{1+\rho^{2\alpha}_n |f(z_n+\rho_n\xi)|^2} \\&\leq \frac{\rho^{1+\alpha}_n \left(|f_z(z_n+\rho_n\xi)| + |f_{\overline{z}}(z_n+ \rho_n\xi)| \right)}{1+ |f(z_n+\rho_n\xi)|^2} \\&= \rho^{1+ \alpha}_n f^{\#}(z_n+\rho_n\xi)\\&\leq \frac{\rho^{1+\alpha}_n} {1-|z_n|-\rho_n |\xi|} \left(1 -|z_n+\rho_n\xi|^2\right) f^{\#}(z_n+\rho_n\xi)
\end{align*}
and $f$ is normal, we have
\begin{align*}
	\left(1 -|z_n+\rho_n\xi|^2\right) f^{\#}(z_n+\rho_n\xi)<\infty.
\end{align*}
Therefore, it is easy to see that $F^{\#}_n(\xi)\rightarrow 0$ as $n\rightarrow\infty$ and thus, $F^{\#}(\xi)=0$ for all $\xi\in \mathbb{C}$, so that $F(\xi)$ is a constant, which is a contradiction. This completes the proof.
\end{proof}

Now, we extend the theorem of Lohwater and Pommerenke \cite[Theorem 1]{Lohwater-Pommerenke-AASF-1973}, a well-known tool in the theory of normal families, to the case of normal logharmonic mappings. If we choose $\alpha = 0$ in Theorem \ref{th-2.2}, we obtain a result corresponding to the Lohwater-Pommerenke theorem. However, the proof presented here follows a different approach from Theorem \ref{th-2.2}. Additionally, this result provides a necessary and sufficient condition for normality.
\begin{thm}
A non-constant function $f(z)=z|z|^{2\beta}h(z) \overline{g(z)}$ logharmonic mapping in $\mathbb{D}$ is normal if and only if there do not exist sequences of points $\{z_n\}\subset\mathbb{D}$ and of real numbers $\{\rho_n\}$ with $\rho_n>0$ and $\rho_n\rightarrow 0$ as $n\rightarrow\infty$ such that the functions 
\begin{align*}
	F_n(\xi):= f(z_n+\rho_n\xi)
\end{align*}
converges locally uniformly in $\mathbb{C}$ to a non-constant logharmonic mapping $F(\xi)$ satisfying $F^{\#}(\xi)\leq F^{\#}(0)= 1$.
\end{thm}
\begin{proof}
Suppose that $f$ is not normal. Then there exists a sequence $\{z^*_n\}\subset\mathbb{D}$ with $z^*_n\rightarrow 1$ as $n\rightarrow\infty$ such that 
\begin{align*}
	\left(1-|z^*_n|^2\right)\frac{|f_z(z^*_n)|+|f_{\overline{z}}(z^*_n)|}{1+|f(z^*_n)|^2}\rightarrow\infty \;\;\;\; \mbox{as}\;\;\;\;n\rightarrow\infty.
\end{align*}
Let $\{r_n\}$ be a sequence such that $|z^*_n|<r_n<1$ and
\begin{align}\label{eq-2.11}
	\left(1-\frac{|z^*_n|^2}{r^2_n}\right)\frac{|f_z(z^*_n)|+|f_{\overline{z}}(z^*_n)|}{1+|f(z^*_n)|^2}\rightarrow\infty \;\;\;\; \mbox{as}\;\;\;\;n\rightarrow\infty.
\end{align}
Moreover, we choose $\{z_n\}$ such that
\begin{align}\label{eq-2.12}
	M_n:=\max_{|z|<r_n}\left(1-\frac{|z|^2}{r^2_n}\right)f^{\#}(z) =\left(1-\frac{|z_n|^2}{r^2_n}\right)f^{\#}(z_n),
\end{align}
the maximum exists because $f^{\#}(z)$ is continuous in $\{|z|\leq r_n\}$. Since $|z^*_n|< r_n$, it follows from \eqref{eq-2.11} and \eqref{eq-2.12} that $M_n\rightarrow\infty$ as $n\rightarrow\infty$. Now, we set 
\begin{align}\label{eq-2.13}
	\rho_n=\frac{1}{M_n} \left(1-\frac{|z_n|^2}{r^2_n}\right) =\frac{1}{f^{\#}(z_n)},
\end{align}
then it is easy to see that
\begin{align}\label{eq-2.14}
	\frac{\rho_n}{1-|z_n|}\leq \frac{\rho_n}{r_n-|z_n|}&= \frac{(r^2_n -|z_n|^2)}{M_n r^2_n (r_n -|z_n|)}\\&\nonumber=\frac{r_n +|z_n|}{r^2_n M_n}\leq \frac{2}{r_n M_n}\rightarrow 0 \;\;\;\;\mbox{as}\;\;\;\; n\rightarrow\infty.
\end{align}
Let the functions $F_n(\xi)=f(z_n+\rho_n\xi)$ is defined for $|\xi|<R_n=\frac{1-|z_n|}{\rho_n}$. From \eqref{eq-2.14} we also note that $R_n\rightarrow\infty$ as $n\rightarrow\infty$. Now, we have
\begin{align*}
	F^{\#}_n(\xi)= \frac{\rho_n\left(|f_z(z_n+\rho_n\xi)| + |f_{\overline{z}}(z_n+\rho_n\xi)|\right)}{1+ |f(z_n+\rho_n\xi)|^2}.
\end{align*}
In view of \eqref{eq-2.13} a simple computation shows that
\begin{align}\label{eq-2.15}
	F^{\#}_n(0)= \frac{\rho_n\left(|f_z(z_n)| + |f_{\overline{z}} (z_n)|\right)}{1+ |f(z_n)|^2}=\rho_n f^{\#}(z_n)=1.
\end{align}
Now, we apply Theorem \ref{th-2.1} to show that the sequence $\{F_n(\xi)\}$ is normal. If $|\xi|\leq R\leq R_n$, then by \eqref{eq-2.13}, we see that
\begin{align}\label{eq-2.16}
	F^{\#}_n(\xi)=\rho_n f^{\#}(z_n+\rho_n\xi)&\leq \frac{\rho_n M_n r^2_n}{r^2_n -|z_n+\rho_n\xi|^2} \\&\nonumber=\frac{r^2_n -|z_n|^2}{r^2_n -|z_n+\rho_n\xi|^2} \\&\nonumber\leq \frac{(r_n +|z_n|)}{(r_n +|z_n|-\rho_n R)}\cdot\frac{(r_n -|z_n|)}{(r_n -|z_n|-\rho_n R)}
\end{align} 
which tends to $1$ as $n\rightarrow\infty$ by \eqref{eq-2.14}, for each fixed $R$. Hence $\{F_n(\xi)\}$ is normal sequence. We may assume that $\{F_n(\xi)\}$ converges locally uniformly in $\mathbb{C}$. Then, the limit function $\{F(\xi)\}$ is logharmonic in $\mathbb{C}$. In view of \eqref{eq-2.15} and \eqref{eq-2.16} it follows that $F^{\#}(\xi)\leq F^{\#}(0)= 1\neq 0$, which implies that $F$ is a non-constant logharmonic mapping. \vspace{2mm}

Next, we prove the necessary part of the theorem. Let $f$ be normal in $\mathbb{D}$. Again, we recall that the functions $F_n(\xi)$ given by
\begin{align*}
	F_n(\xi)=f(z_n+\rho_n\xi)
\end{align*}
are defined for $|\xi|<\frac{1-|z_n|}{\rho_n}$, and by \eqref{eq-2.14} , we see that $\frac{\rho_n}{1-|z_n|}\rightarrow 0$ as $n\rightarrow \infty$. Again, we observe that
\begin{align*}
	\frac{\rho_n}{1-|z_n|-\rho_n |\xi|}\rightarrow 0 \;\;\;\;\mbox{as} \;\;\;\;n\rightarrow \infty,\;\;\mbox{for}\;\; |\xi|<\frac{1-|z_n|}{\rho_n}.
\end{align*}
Since
\begin{align*}
	F^{\#}_n(\xi)&=\rho_n \frac{\left(|f_z(z_n+\rho_n\xi)| + |f_{\overline{z}}(z_n+\rho_n\xi)|\right)}{1+ |f(z_n+\rho_n\xi)|^2} \\&\leq \frac{\rho_n} {1-|z_n|-\rho_n |\xi|}\cdot \left(1 -|z_n+\rho_n\xi|^2\right) \cdot\frac{\left(|f_z(z_n+\rho_n\xi)| + |f_{\overline{z}}(z_n+\rho_n\xi)|\right)}{1+ |f(z_n+\rho_n\xi)|^2}
\end{align*}
and $f$ is normal, we have
\begin{align*}
	\left(1 -|z_n+\rho_n\xi|^2\right) \frac{\left(|f_z(z_n+\rho_n\xi)| + |f_{\overline{z}}(z_n+\rho_n\xi)|\right)}{1+ |f(z_n+ \rho_n\xi)|^2}<\infty.
\end{align*}
Therefore, it is easy to see that $F^{\#}_n(\xi)\rightarrow 0$ as $n\rightarrow\infty$ and thus, $F^{\#}(\xi)=0$ for all $\xi\in \mathbb{C}$, so that $F(\xi)$ is a constant, which is a contradiction. This completes the proof. 
\end{proof}

The following result essentially demonstrates that the normality of a logharmonic mapping $f$ in $\mathbb{D}$ can be deduced even when the spherical derivative $f^{\#}(z)$ of logharmonic mapping $f$ is bounded away from zero. This aligns with the result of Grahl and Nevo \cite[Theorem 1]{Grahl-Nevo-JdMS-2012} for meromorphic functions. As an application of Theorem \ref{th-2.2}, we derive the following result:
\begin{thm}
A class $\mathcal{F}$ of logharmonic mapping $f(z)=z|z|^{2\beta}h(z) \overline{g(z)}$ in the unit disk $\mathbb{D}$ and and let $\epsilon> 0$ be a real number. If, for all $f\in\mathcal{F}$ and $z\in\mathbb{D}$,  $f^{\#}(z)>\epsilon$, then $f$ is a normal logharmonic mapping. 
\end{thm}
\begin{proof}
Suppose that $f$ is a sense-preserving logharmonic mapping in $\mathbb{D}$ which is not normal. By Theorem \ref{th-2.2}, there exists a sequence $\{z_n\}\subset\mathbb{D}$ and a sequence of positive real numbers $\rho_n$ such that $z_n\rightarrow 1$  and $\rho_n\rightarrow 0$, and a non-constant sense-preserving logharmonic mapping $F$ in $\mathbb{C}$ such that
the sequence
\begin{align*}
	F_n(\xi)=\rho^2_n f(z_n+\rho_n\xi)
\end{align*}
converges locally uniformly to $F$ as $n\rightarrow\infty$, with $F^{\#}(\xi)\leq 1$ on $\mathbb{C}$ and $F^{\#}(0)= 1$. Then there exists $\xi_0\in\mathbb{C}$ such that $|F(\xi_0)|>0$ and so, for
sufficiently large $n$, $|F_n(\xi_0)|\neq 0$. Threfore, we observe that
\begin{align*}
	F^{\#}_n(\xi_0)&= \frac{\rho^3_n\left(|f_z(z_n+\rho_n\xi_0)| + |f_{\overline{z}}(z_n+\rho_n\xi_0)|\right)}{1+\rho^4_n |f(z_n+\rho_n\xi_0)|^2}\\&= \frac{\rho^3_n\left(|f_z(z_n+\rho_n\xi_0)| + |f_{\overline{z}}(z_n+\rho_n\xi_0)|\right)}{\rho^4_n |f(z_n+\rho_n\xi_0)|^2}\cdot\frac{\rho^4_n |f(z_n+\rho_n\xi_0)|^2}{1+\rho^4_n |f(z_n+\rho_n\xi_0)|^2} \\&\geq \frac{f^{\#}(z_n+\rho_n\xi_0)}{\rho_n} \cdot\frac{ |F_n(\xi_0)|^2}{1+|F_n(\xi_0)|^2}\\&> \frac{\epsilon}{\rho_n} \cdot\frac{ |F_n(\xi_0)|^2}{1+|F_n(\xi_0)|^2} \rightarrow\infty \;\;\;\;\mbox{as}\;\;\;\;n\rightarrow\infty.
\end{align*}
This is a contradiction to the fact that $F^{\#}(\xi)\leq 1$ for each $\xi\in\mathbb{C}$. Thus, $f$ is a normal logharmonic mapping. This completes the proof.
\end{proof}

\noindent{\bf Acknowledgment:} The authors are deeply grateful to the anonymous referee(s) for their detailed comments and valuable suggestions, which have significantly improved the presentation of the paper.

\section{\bf Declaration}
\noindent\textbf{Compliance of Ethical Standards:}\\

\noindent\textbf{Conflict of interest.} The authors declare that there is no conflict  of interest regarding the publication of this paper.\vspace{1.5mm}

\noindent\textbf{Data availability statement.}  Data sharing is not applicable to this article as no datasets were generated or analyzed during the current study.\vspace{1.5mm}

\noindent\textbf{Declaration of Fundings.} The authors declare that no funds, grants, or other support were received during the preparation of this manuscript.

\end{document}